\documentclass[12pt]{amsart}
\usepackage{amsmath,amssymb,amsbsy,amsfonts,latexsym,amsopn,amstext,cite,
                                               amsxtra,euscript,amscd,bm}
\usepackage{url}

\usepackage{mathrsfs}

\usepackage{color}
\usepackage[colorlinks,linkcolor=blue,anchorcolor=blue,citecolor=blue,backref=page]{hyperref}
\usepackage{color}
\usepackage{graphics,epsfig}
\usepackage{graphicx}
\usepackage{float}
\usepackage{epstopdf}
\hypersetup{breaklinks=true}

\usepackage[np]{numprint}
\npdecimalsign{\ensuremath{.}}

\usepackage{bibentry}

\usepackage[english]{babel}
\usepackage{mathtools}
\usepackage{todonotes}
\usepackage{url}
\usepackage[colorlinks,linkcolor=blue,anchorcolor=blue,citecolor=blue,backref=page]{hyperref}

\usepackage[norefs,nocites]{refcheck}

\def\Degf{e}

\def\cont{\mathrm{cont}}

\begin{document}

\newcommand{\bs}{\boldsymbol}
\def \a{\alpha} \def \b{\beta} \def \d{\delta} \def \e{\varepsilon} \def \g{\gamma} \def \k{\kappa} \def \l{\lambda} \def \s{\sigma} \def \t{\theta} \def \z{\zeta}

\newcommand{\mb}{\mathbb}

\newtheorem{theorem}{Theorem}
\newtheorem{lemma}[theorem]{Lemma}
\newtheorem{claim}[theorem]{Claim}
\newtheorem{cor}[theorem]{Corollary}
\newtheorem{conj}[theorem]{Conjecture}
\newtheorem{prop}[theorem]{Proposition}
\newtheorem{definition}[theorem]{Definition}
\newtheorem{question}[theorem]{Question}
\newtheorem{example}[theorem]{Example}
\newcommand{\hh}{{{\mathrm h}}}
\newtheorem{remark}[theorem]{Remark}

\numberwithin{equation}{section}
\numberwithin{theorem}{section}
\numberwithin{table}{section}
\numberwithin{figure}{section}

\def\sssum{\mathop{\sum\!\sum\!\sum}}
\def\ssum{\mathop{\sum\ldots \sum}}
\def\iint{\mathop{\int\ldots \int}}

\newcommand{\diam}{\operatorname{diam}}

\def\squareforqed{\hbox{\rlap{$\sqcap$}$\sqcup$}}
\def\qed{\ifmmode\squareforqed\else{\unskip\nobreak\hfil
\penalty50\hskip1em \nobreak\hfil\squareforqed
\parfillskip=0pt\finalhyphendemerits=0\endgraf}\fi}

\newfont{\teneufm}{eufm10}
\newfont{\seveneufm}{eufm7}
\newfont{\fiveeufm}{eufm5}
%
%
\newfam\eufmfam
     \textfont\eufmfam=\teneufm
\scriptfont\eufmfam=\seveneufm
     \scriptscriptfont\eufmfam=\fiveeufm
%
%
\def\frak#1{{\fam\eufmfam\relax#1}}

\newcommand{\bflambda}{{\boldsymbol{\lambda}}}
\newcommand{\bfmu}{{\boldsymbol{\mu}}}
\newcommand{\bfxi}{{\boldsymbol{\eta}}}
\newcommand{\bfrho}{{\boldsymbol{\rho}}}

\def\eps{\varepsilon}

\def\fK{\mathfrak K}
\def\fT{\mathfrak{T}}
\def\fL{\mathfrak L}
\def\fR{\mathfrak R}
\def\fQ{\mathfrak Q}

\def\fA{{\mathfrak A}}
\def\fB{{\mathfrak B}}
\def\fC{{\mathfrak C}}
\def\fM{{\mathfrak M}}
\def\fS{{\mathfrak  S}}
\def\fU{{\mathfrak U}}

\def\sssum{\mathop{\sum\!\sum\!\sum}}
\def\ssum{\mathop{\sum\ldots \sum}}
\def\dsum{\mathop{\quad \sum \qquad \sum}}
\def\iint{\mathop{\int\ldots \int}}
 
\def\T {\mathsf {T}}
\def\Tor{\mathsf{T}_d}
\def\Tore{\widetilde{\mathrm{T}}_{d} }

\def\sM {\mathsf {M}}
\def\sL {\mathsf {L}}
\def\sK {\mathsf {K}}
\def\sP {\mathsf {P}}

\def\ss{\mathsf {s}}

\def \balpha{\bm{\alpha}}
\def \bbeta{\bm{\beta}}
\def \bgamma{\bm{\gamma}}
\def \bdelta{\bm{\delta}}
\def \bzeta{\bm{\zeta}}
\def \blambda{\bm{\lambda}}
\def \bchi{\bm{\chi}}
\def \bphi{\bm{\varphi}}
\def \bpsi{\bm{\psi}}
\def \bxi{\bm{\xi}}
\def \bnu{\bm{\nu}}
\def \bomega{\bm{\omega}}

\def \bell{\bm{\ell}}

\def\eqref#1{(\ref{#1})}

\def\vec#1{\mathbf{#1}}

\newcommand{\abs}[1]{\left| #1 \right|}

\def\Zq{\mathbb{Z}_q}
\def\Zqx{\mathbb{Z}_q^*}
\def\Zd{\mathbb{Z}_d}
\def\Zdx{\mathbb{Z}_d^*}
\def\Zf{\mathbb{Z}_f}
\def\Zfx{\mathbb{Z}_f^*}
\def\Zp{\mathbb{Z}_p}
\def\Zpx{\mathbb{Z}_p^*}
\def\cM{\mathcal M}
\def\cE{\mathcal E}
\def\cH{\mathcal H}

\def\le{\leqslant}

\def\ge{\geqslant}

\def\sfB{\mathsf {B}}
\def\sfC{\mathsf {C}}
\def\sfS{\mathsf {S}}
\def\sfI{\mathsf {I}}
\def\L{\mathsf {L}}
\def\FF{\mathsf {F}}

\def\sE {\mathscr{E}}
\def\sS {\mathscr{S}}

\def\cA{{\mathcal A}}
\def\cB{{\mathcal B}}
\def\cC{{\mathcal C}}
\def\cD{{\mathcal D}}
\def\cE{{\mathcal E}}
\def\cF{{\mathcal F}}
\def\cG{{\mathcal G}}
\def\cH{{\mathcal H}}
\def\cI{{\mathcal I}}
\def\cJ{{\mathcal J}}
\def\cK{{\mathcal K}}
\def\cL{{\mathcal L}}
\def\cM{{\mathcal M}}
\def\cN{{\mathcal N}}
\def\cO{{\mathcal O}}
\def\cP{{\mathcal P}}
\def\cQ{{\mathcal Q}}
\def\cR{{\mathcal R}}
\def\cS{{\mathcal S}}
\def\cT{{\mathcal T}}
\def\cU{{\mathcal U}}
\def\cV{{\mathcal V}}
\def\cW{{\mathcal W}}
\def\cX{{\mathcal X}}
\def\cY{{\mathcal Y}}
\def\cZ{{\mathcal Z}}
\newcommand{\rmod}[1]{\: \mbox{mod} \: #1}

\def\cg{{\mathcal g}}

\def\vy{\mathbf y}
\def\vr{\mathbf r}
\def\vx{\mathbf x}
\def\va{\mathbf a}
\def\vb{\mathbf b}
\def\vc{\mathbf c}
\def\ve{\mathbf e}
\def\vf{\mathbf f}
\def\vg{\mathbf g}
\def\vh{\mathbf h}
\def\vk{\mathbf k}
\def\vm{\mathbf m}
\def\vz{\mathbf z}
\def\vu{\mathbf u}
\def\vv{\mathbf v}

\def\e{{\mathbf{\,e}}}
\def\ep{{\mathbf{\,e}}_p}
\def\eq{{\mathbf{\,e}}_q}

\def\Tr{{\mathrm{Tr}}}
\def\Nm{{\mathrm{Nm}}}

 \def\SS{{\mathbf{S}}}

\def\lcm{{\mathrm{lcm}}}

 \def\0{{\mathbf{0}}}

\def\({\left(}
\def\){\right)}
\def\l|{\left|}
\def\r|{\right|}
\def\fl#1{\left\lfloor#1\right\rfloor}
\def\rf#1{\left\lceil#1\right\rceil}
\def\sumstar#1{\mathop{\sum\vphantom|^{\!\!*}\,}_{#1}}

\def\mand{\qquad \mbox{and} \qquad}

\def\tblue#1{\begin{color}{blue}{{#1}}\end{color}}




\hyphenation{re-pub-lished}

\mathsurround=1pt

\def\bfdefault{b}

\def \F{{\mathbb F}}
\def \K{{\mathbb K}}
\def \N{{\mathbb N}}
\def \Z{{\mathbb Z}}
\def \P{{\mathbb P}}
\def \Q{{\mathbb Q}}
\def \R{{\mathbb R}}
\def \C{{\mathbb C}}
\def\Fp{\F_p}
\def \fp{\Fp^*}

 \def \xbar{\overline x}

\title[Glasner property for matrices with polynomial entries]{On the Glasner property for matrices with polynomial entries}

 \author[I. E. Shparlinski] {Igor E. Shparlinski}
\address{Department of Pure Mathematics, University of New South Wales,
Sydney, NSW 2052, Australia}
\email{igor.shparlinski@unsw.edu.au}

\begin{abstract}   We obtain a new bound in the uniform version of  the Glasner property for matrices with polynomial entries, 
improving that of K.~Bulinski and A.~Fish (2021). This improvement is based on a more careful examination of complete rational exponential sums with polynomials and can perhaps be used in other questions of the similar spirit. 
\end{abstract}


\maketitle

%

\section{Introduction}

\subsection{Set-up and motivation} 
Let $\T = \R/\Z$ be the unit torus which we identify with the half-open interval $[0,1)$.   Glasner~\cite{Glas} has shown that for any infinite set $\cX \in \T$ and any $\varepsilon > 0$ 
one can find an $n \in \N$ such that the dilation $n\cX$ is {\it $\varepsilon$-dense\/} in $\T$, that is, for some   $n \in \N$  we have
$$
\sup_{\zeta \in \T}\, \inf_{x \in \cX} |\zeta  - n x |  \le \varepsilon.
$$

The result has been extended and improved in several directions including polynomial sequences of dilations $f(n) \cX$ 
with $f \in \Z[X]$ and also to multidimensional analogues  for sets $\cX \subseteq \Tor$, where 
$$
\Tor = \( \R/\Z\)^d
$$
for some integer $d\ge 1$, we refer to~\cite{BuFi} for a survey of previous results and further references. 

Here we continue considering the same scenario as in the recent work of Bulinski and Fish~\cite{BuFi}  and consider 
dilations $A(n)\cX$ of a set $\cX \subseteq \Tor$ by an integer polynomial matrix
 \begin{equation}\label{eq: matrix A}
A(X) = \(a_{r,s}(X)\)_{r,s=1}^d \in \Z[X]^{d\times d}.
\end{equation}

Recently, improving and generalising several previous results,  Bulinski and Fish~\cite{BuFi}  
have shown a version  of the multidimensional Glasner theorem~\cite{Glas} with 
polynomial matrix  dilations, where the set $\cX \subseteq \Tor$ can be finite.
Namely, provided that the matrix $A$ satisfies some natural necessary condition (see 
Theorem~\ref{thm:k bound}  below), 
for any  $\varepsilon>0$ there is some $k_{d, A} (\varepsilon)$ such that if $\cX \subseteq \Tor$
is of cardinality $\#\cX \ge k_{d, A} (\varepsilon)$ then for some $n\in \N$ the dilation $A(n) \cX$ is
$\varepsilon$-dense in $\T$, that is, for some   $n \in \N$  we have
$$
\sup_{\bzeta \in \Tor} \, \inf_{\vx \in \cX} \|\bzeta  - A(n) \vx \|  \le \varepsilon, 
$$
where $\|  \cdot \|$ is the distance on $\Tor$ induced by the Euclidean norm. 
In fact~\cite[Theorem~2.8]{BuFi}  gives the following   bound 
 \begin{equation}\label{eq: bound k B-F}
 k_{d, A} (\varepsilon) \le c(d, \Degf) H^{d(d+1)} \varepsilon^{-d(d+1)\(2 \Degf +1\)  + o(1)}
 \qquad \text{as}\ \varepsilon \to 0,
\end{equation}
where $\Degf$ and $H$ are the largest degree and  absolute value of non-constant 
coefficients of polynomials $a_{r,s}$ in~\eqref{eq: matrix A}, respectively, and $c(d, \Degf)$ depends only on $d$ and $\Degf$. 

\subsection{New bound}  
As in~\cite{BuFi} we note that we can always assume that the matrix~\eqref{eq: matrix A} satisfies 
 \begin{equation}\label{eq: const term}
a_{r,s}(0)= 0, \qquad  r,s=1, \ldots, d.
\end{equation}

As usual, for a vector $\vu\in \R^d$ we denote by $\vu^t$ the transposed vector. 

 \begin{theorem}
\label{thm:k bound} 
Suppose that  the matrix $A(X)$ as in~\eqref{eq: matrix A} satisfies~\eqref{eq: const term} and is
such that for any non-zero vectors $\vu, \vv \in \Z^d$ we have $\vu^t A(X) \vv \ne 0$. Then 
$$
 k_{d, A} (\varepsilon) \le c(d, \Degf) H^{(3d+1)/2+o(1)}  \varepsilon^{-d (2d+1) \Degf    -(7d+1)/2 +o(1)}, \qquad \text{as}\ \varepsilon \to 0,
 $$
 where $\Degf$ and $H$ are the largest degree and absolute value of the 
coefficients of polynomials $a_{r,s}(X)$, $r,s=1, \ldots, d$, respectively. 
\end{theorem}

To compare Theorem~\ref{thm:k bound} with the bound~\eqref{eq: bound k B-F} we notice that 
$$
\frac{3d+1}{2} \le d(d+1) \mand 
d (2d+1) \Degf  +   \frac{7d+1}{2}  \le d(d+1)\(2 \Degf +1\)
$$
for all $d \ge 1$ and $e \ge 2$. 

\subsection{Ideas behind the proof} 
The proof of~\eqref{eq: bound k B-F} in~\cite{BuFi}  is based on the classical  {\it Hua bound\/}  (see~\eqref{eq:Hua} below) of complete 
rational exponential sums (for a proof see, for example,~\cite[Theorem~7.1]{Vau}). This bound has also been used in~\cite{KeLe}. Generally speaking, the Hua bound~\eqref{eq:Hua} is tight and cannot be improved for arbitrary moduli $q$. However the set of moduli for which it is optimal is rather sparse, which is the idea we exploit here. More precisely,   we use more refined information about  complete rational exponential sums (see Lemma~\ref{lem:d-power-factor}) and 
some well-known results about the arithmetic structure of integers (see~\eqref{eq:i-full}). This allows us to  improve the exponent of~\eqref{eq: bound k B-F}. 

Although the improved bound is perhaps still far from the optimal bound, we believe that the technique we employ
deserves to be known better and can also be used for some other problems.

\section{Preparations} 

\subsection{Notation and conventions} 

Throughout the paper, the notations $U = O(V)$, 
$U \ll V$ and $ V\gg U$  are equivalent to $|U|\leqslant c V$ for some positive constant $c$, 
which throughout the paper may depend on the dimension $d$ and the degree $\Degf$.

We always use $\Degf$  for the largest degree and use $H$ for the  absolute value of the 
coefficients of polynomials $a_{r,s}(X)$, $r,s=1, \ldots, d$ of the matrix~\eqref{eq: matrix A}, 
which we always assume to satisfy~\eqref{eq: const term}.

Thus we always suppress the dependence on $d$ and $\Degf$ in the `$\ll$', however we give implicit
(albeit not optimised) estimates in terms of $H$. 

For any quantity $V> 1$ we write $U = V^{o(1)}$ (as $V \rightarrow \infty$) to indicate a function of $V$ which 
satisfies $ V^{-\delta} \le |U| \le V^\delta$ for any $\delta> 0$, provided $V$ is large enough. One additional advantage 
of using $V^{o(1)}$ is that it absorbs $\log V$ and other similar quantities without changing  the whole 
expression. 

We identify $\Tor$ with the unit cube $[0,1)^d$ and thus treat elements of $\Tor$ as real numbers.

Given a vector $\va = (a_1, \ldots, a_\nu) \in \Z^\nu$ and $q\in \N$ we write 
$$
\gcd(\va, q) = \gcd(a_1, \ldots, a_\nu,q). 
$$

For a polynomial 
$$
f(X) = f_\Degf X^\Degf + \ldots + f_1 X + f_0 \in \Z[X]
$$
and $q \in \N$, 
we define  the {\it $q$-content\/}  $\cont_q(f)$ of $f$ by 
$$\cont_q(f) = \gcd(f_1, \ldots, f_\Degf,q). 
$$

We say that   $f$ is {\it $q$-primitive\/} if  $\cont_q(f)=1$. 

Finally, we denote 
$$
\e_q(z) = \exp(2 \pi z/q).
$$
\subsection{Reduction to bounds of complete rational  exponential sums} 
It is shown in the proof of~\cite[Theorem~2.8]{BuFi} that we can assume that 
$$
\vx_i - \vx_j \in \Q, \qquad i,j=1, \ldots, k.
$$  
Furthermore, since additive shifts of $\cX$ do not change the property of $A(n)\cX$ 
being $\varepsilon$-dense it is sufficient to consider the case $\cX   \subseteq \Tor\cap \Q^d$. 

We now assume that the set 
\begin{equation}
\label{eq:set X} 
\cX = \{\vx_1, \ldots, \vx_k\}  \subseteq \Tor\cap \Q^d
\end{equation}
is fixed and for an integer $q\ge 1$ we denote by $h_q$ 
the number of pairs $(i,j)$, $1 \le i, j \le k$, such that $q$ is the smallest integer with  $q\(\vx_i - \vx_j\) \in \Z^d$ 
(or, alternatively,  $q\(\vx_i - \vx_j\) = {\mathbf 0}$ if considered as elements of $\Tor$).

Also for an integer $M\ge 1$ we consider the set 
$$
\cB(M) = \{ \vm  \in \Z^d:~ \vm \ne  {\mathbf 0}\} \cap [-M,M]^d. 
$$


As in~\cite{BuFi}, our argument is based on the following inequality which is contained in~\cite[Proposition~2.6]{BuFi}, and 
which in turn follows from~\cite[Proposition~2]{KeLe}. 

\begin{lemma} 
\label{lem:Bad Sets}
If a set $\cX$ as in~\eqref{eq:set X} is such that  for any $n\in \N$ the dilation $A(n) \cX$ is not 
$\varepsilon$-dense in $\Tor$ then for $M= \fl{d/\varepsilon}$ we have 
$$
k^2 \ll \varepsilon^{-d} \sum_{\vm  \in \cB(M)} \sum_{q=1}^\infty \frac{h_q}{q} \left|\sum_{n=1}^q \eq\(\vm^t A(n) \vb_q\) \right|
+ \varepsilon^{-d}M^d k, 
$$ 
where $\vb_{q}\in \Z^d$ some integer vectors with  $\gcd(\vb_q, q) = 1$. 
\end{lemma}

Hence, to use Lemma~\ref{lem:Bad Sets} we need:
\begin{itemize}
\item estimate the coefficients $h_q$, see the bound~\eqref{eq:sum hq} 
and Lemma~\ref{lem:hq indiv}; 
\item estimate the content $\cont_q \(\vm^t A(X) \vb_{q}\)$ of the polynomials
in the exponential sums see Lemma~\ref{lem:hq tail};
\item use bounds of complete  rational exponential sums for $q$-primi\-tive polynomials. 
see Corollary~\ref{cor:d-power-factor gcd}. 
\end{itemize}

\subsection{Some arithmetic estimates} 
\label{sec:arith estim} 
First we estimate the coefficients $h_q$ for a given set~\eqref{eq:set X}.
 Hence, we   have  the following  trivial identity
\begin{equation}
\label{eq:sum hq} 
\sum_{q =1}^\infty h_q = k^2 .
\end{equation}

Using  the argument in the proof of~\cite[Proposition~1]{KeLe}, which is  also used 
in the proof of~\cite[Proposition~2.7]{BuFi}, we immediately obtain the   
following bound. 

\begin{lemma} 
\label{lem:hq indiv}
For any $q \in \N$ we have $h_q \le k q^d$.
\end{lemma}

\begin{proof} Clearly, for each $i =1, \ldots, k$ there are at most $q^d$ vectors $\vx \in \Tor$ with  $q\(\vx_i - \vx\) \in \Z^d$.
\end{proof}

We make use of the following upper bound on  $\cont_q \(\vm^t A(X) \vb_{q}\)$ in Lemma~\ref{lem:Bad Sets}, 
which follows from~\cite[Corollary~2.4]{BuFi}.

\begin{lemma} 
\label{lem:hq tail}
For  any real $M \ge 1$ uniformly over vectors $\vm \in \cB(M)$ and vectors $\vb\in \Z^d$ with $\gcd(\vb,q)=1$, we have 
$$
\cont_q \(\vm^t A(X) \vb\) \ll (HM)^d. 
$$
\end{lemma}

Given an integer $\nu \ge 2$, an integer number $n$ is called
\begin{itemize}
\item  {\it $\nu$-th power free\/} if   any prime number $p \mid n$ satisfies $p^\nu \nmid n$; 

\item  {\it  $\nu$-th power  full\/} if any prime number  $p \mid n$ satisfies $p^\nu \mid n$. 
\end{itemize}
We note that $1$ is both $\nu$-th  power free and $\nu$-th  power full for any $\nu$.   

For any integer $i\ge 2$ it is  convenient to denote   
$$
\cF_{\nu}=\{n \in \N:~  \text{$n$ is $\nu$-th power full}\} \quad \text{and} \quad
\cF_{\nu}(x)= \cF_\nu\cap [1,x].
$$ 
The classical result of  Erd{\H o}s and  Szekeres~\cite{ErdSz} gives an asymptotic 
formula for the cardinality of $\cF_{\nu}(x)$ which we present here in a very relaxed form 
as the upper bound 
\begin{equation}
\label{eq:i-full}
\# \cF_{\nu}(x) \ll x^{1/\nu} . 
\end{equation}

\subsection{Bounds of complete rational exponential sums} 

For $q\in \N$ and $\vf = \(f_1, \ldots, f_\Degf\)$, we denote
$$
S_{\Degf,q}(\vf) =\sum_{n=1}^{q}\e_q(f_1n+\ldots+f_\Degf n^\Degf ).
$$

Our new tool is the following bound on $|S_{\Degf,q}(\vf)|$ which is derived in~\cite{BCS} from the classical {\it Weil and Hua bounds\/},  see, for example,~\cite[Theorem~5.38]{LN}, 
combined with,~\cite[Equation~(2.5)]{CPR} and the results of,~\cite{DiQi,Stech} giving  a slight improvement of the 
Hua bound (see also~\cite[Theorem~7.1]{Vau}).

\begin{lemma}
\label{lem:d-power-factor}
Write  an integer $q\ge 1$ as 
$q=q_2 \ldots q_\Degf $ such that
\begin{itemize} 
\item $q_2\ge 1$ is  cube  free,
\item $q_i$ is $i$-th power full but $(i+1)$-th power free when $3 \le i \le \Degf -1$,
\item $q_\Degf$ is   $\Degf $-th power full, 
\end{itemize}
with $\gcd(q_i,q_j) = 1$, $2 \le i < j \le \Degf $.
For  $\vf =  \(f_1, \ldots, f_\Degf \)\in \Z^\Degf$ with 
$$
\gcd\(q, f_1, \ldots, f_\Degf \) = 1, 
$$ 
we have 
$$
|S_{\Degf ,q}(\vf)| \le  \prod_{i=2}^\Degf  q_i^{1-1/i} q^{o(1)}. 
$$ 
\end{lemma}

We remark, that in~\cite{BuFi} the bound 
\begin{equation}
\label{eq:Hua} 
|S_{\Degf ,q}(\vf)| \le   q^{1-1/\Degf  +o(1)}
\end{equation}
has been used. So our improvement comes from using Lemma~\ref{lem:d-power-factor}
together with a classification of moduli $q$ for which it gives an improvement of~\eqref{eq:Hua}. 

We now immediately derive from Lemma~\ref{lem:d-power-factor} the following more general bound. 

\begin{cor}
\label{cor:d-power-factor gcd}
Write  an integer $q\ge 1$ as 
$q=q_2 \ldots q_\Degf $ such that
\begin{itemize} 
\item $q_2\ge 1$ is  cube  free,
\item $q_i$ is $i$-th power full but $(i+1)$-th power free when $3 \le i \le \Degf -1$,
\item $q_\Degf $ is   $\Degf $-th power full, 
\end{itemize}
with $\gcd(q_i,q_j) = 1$, $2 \le i < j \le \Degf $.
For   $\vf =  \(f_1, \ldots, f_\Degf \)\in \Z^\Degf$  with 
$$
\gcd\(q, f_1, \ldots, f_\Degf \) =  s, 
$$ 
we have 
$$
|S_{\Degf ,q}(\vf)| \le q^{1+o(1)}  \prod_{i=2}^\Degf  \(q_i/\gcd(q_i,s)\)^{-1/i} .
$$ 
\end{cor}

\section{Proof of Theorem~\ref{thm:k bound}} 

\subsection{Preliminary split} 
Let $\cX$ be a set  as in~\eqref{eq:set X}  such that  for any $n\in \N$ the dilation $A(n) \cX$ is not 
$\varepsilon$-dense in $\Tor$.

We choose some integer parameter $R\ge 1$ and using Lemma~\ref{lem:Bad Sets} write 
\begin{equation}
\label{eq:S1 S2} 
k^2 \ll  \varepsilon^{-d} \(S_1 + S_2\)+ \varepsilon^{-d}M^d k,
\end{equation}
where 
\begin{align*}
&S_1=\sum_{\vm  \in \cB(M)}  \sum_{q=1}^R \frac{h_q}{q} \left|\sum_{n=1}^q \eq\(\vm^t A(n) \vb_q\) \right|,\\
&S_2=\sum_{\vm  \in \cB(M)}  \sum_{q=R+1}^\infty \frac{h_q}{q} \left|\sum_{n=1}^q \eq\(\vm^t A(n) \vb_q\) \right|.
\end{align*}

\subsection{Bound on $S_1$}  To estimate $S_1$ we first use Lemma~\ref{lem:hq indiv}
and write 
\begin{equation}
\label{eq:S1 Prelim 1} 
S_1 \le k \sum_{\vm  \in \cB(M)}  \sum_{q=1}^R q^{d-1} \left|\sum_{n=1}^q \eq\(\vm^t A(n) \vb_q\) \right|.
\end{equation}

We now note that with  $q_2, \ldots, q_\Degf $,  defined as in Corollary~\ref{cor:d-power-factor gcd}, 
 we have 
\begin{align*}
\left|\sum_{n=1}^q \eq\(\vm^t A(n) \vb_q\) \right| & \le q^{1+o(1)}    \prod_{i=2}^\Degf  \(\frac{q_i}{\gcd(q_i, \cont_q \(\vm^t A(X) \vb_q\)}\)^{-1/i} \\
& \le \cont_q \(\vm^t A(X) \vb_q\)^{1/2}  q^{1+o(1)}    \prod_{i=2}^\Degf  q_i^{-1/i} 
\end{align*}
and using Lemma~\ref{lem:hq tail}, we obtain
\begin{equation}
\label{eq:ExpSumBound} 
\left|\sum_{n=1}^q \eq\(\vm^t A(n) \vb_q\) \right|  \le (HM)^{d/2}  q^{1+o(1)}  \prod_{i=2}^\Degf  q_i^{-1/i} . 
\end{equation}

We now define the sets 
$$
\cG_\nu(x) =  \cF_\nu(x) \backslash  \cF_\nu(x/2)
$$
where, as in Section~\ref{sec:arith estim},  $\cF_\nu(x)$ denotes the set of  $\nu$-th power full positive integers $n \le x$.

Then using the dyadic partition  of the whole domain of possible values 
of $q_2, \ldots, q_\Degf $ as in Corollary~\ref{cor:d-power-factor gcd} in~\eqref{eq:S1 Prelim 1}.
That is, we fix a family $\fQ$ of  at most $O\(\(\log R\)^{e-1}\)$ vectors of real parameters $\(Q_2,  \ldots, Q_\Degf \)$ with 
\begin{equation}
\label{eq:Q2..Qe} 
Q_2\ldots Q_\Degf  \le R
\end{equation}
 and cover the whole range where  $q_2, \ldots, q_\Degf $ 
may vary, by rectangular boxes with 
$$
q_2 \sim Q_2, \ldots,  q_\Degf \sim  Q_\Degf
$$
for $\(Q_2,  \ldots, Q_\Degf \)\in \fQ$.
Recalling the arithmetic structure of $q_2, \ldots,  q_\Degf $ 
we see that in fact 
$$
q_3\in  \cG_3\(Q_3\), \ldots,  q_\Degf \in  \cG_\Degf\( Q_\Degf \).
$$
We now  see that there are some real numbers 
$Q_2,  \ldots, Q_\Degf $ with~\eqref{eq:Q2..Qe} for which  we have 
\begin{align*}
S_1 \ll k \(\log R\)^{e-1} \sum_{\vm  \in \cB(M)}  & \, 
\ssum_{\substack{q_2 \sim Q_2 \\
q_h \in \cG_h(Q_h) , \ 3\le h \le \Degf  \\
\gcd(q_i, q_j) =1, \ 2 \le i < j \le \Degf }}  \(q_2\ldots q_\Degf  \)^{d-1} \\
&\qquad \qquad \quad   \left|\sum_{n=1}^{q_2\ldots q_\Degf } 
\e_{q_2\ldots q_\Degf }\(\vm^t A(n) \vb_{q_2\ldots q_\Degf }\) \right|.
\end{align*}

Hence, recalling~\eqref{eq:Q2..Qe} and applying the bound~\eqref{eq:ExpSumBound},  we obtain 
$$
S_1  \le k  H^{d/2} M^{3d/2} R^{d+o(1)} Q_2^{1/2} \prod_{i=3}^\Degf  \( \# \cG_{i}(Q_i) Q_i^{-1/i}\). 
$$
Finally, using the bound~\eqref{eq:i-full}, we obtain
\begin{equation}
\label{eq:bound S1} 
S_1 \le k  H^{d/2} M^{3d/2} R^{d+o(1)} Q_2^{1/2} \le k  H^{d/2} M^{3d/2} R^{d+1/2+ o(1)} .
\end{equation}

\subsection{Bound on $S_2$} 
To estimate $S_2$ we use the bound of Corollary~\ref{cor:d-power-factor gcd}
in the following crude form 
$|S_{\Degf ,q}(\vf)| \le q^{1-1/\Degf +o(1)}  s^{1/\Degf } $
(which is exactly the bound used in~\cite{BuFi}). 
However it is technically more convenient (but does not affect the final result) to use a slightly more
precise bound 
\begin{equation}
\label{eq:bound Hua} 
S_{\Degf ,q}(\vf) \ll q^{1-1/\Degf }  s^{1/\Degf } 
\end{equation}
without $o(1)$ in the exponent, see~\cite{DiQi,Stech} for explicit evaluations of the implied constant. 

The bound~\eqref{eq:bound Hua}, together with~\eqref{eq:sum hq}   and Lemma~\ref{lem:hq tail} 
implies 
\begin{equation}
\label{eq:bound S2} 
S_2\le (HM)^{1/\Degf}   \sum_{\vm  \in \cB(M)}  \sum_{q=R+1}^\infty h_q q^{-1/\Degf} 
\ll k^2 H^{1/\Degf} M^{d+1/\Degf }  R^{-1/\Degf }. 
\end{equation}

 \subsection{Concluding the proof} 

Substituting the bounds~\eqref{eq:bound S1} and~\eqref{eq:bound S2} in~\eqref{eq:S1 S2}, 
and recalling the value of $M$ in Lemma~\ref{lem:Bad Sets}, 
we obtain
\begin{equation}
\label{eq:bound k} 
k^2 \ll k H^{d/2}  \varepsilon^{ -5d/2} R^{d+1/2+ o(1)} +   k^2 H^{1/\Degf }  \varepsilon^{ -2d-1/\Degf }   R^{-1/\Degf }
\end{equation}
(clearly the term  $\varepsilon^{-d}M^d k$ can be absorbed in the first term). 
Taking 
\begin{equation}
\label{eq:R opt} 
R = C H   \varepsilon^{ -2d\Degf -1} 
\end{equation}
for a sufficiently large constant $C$ (which depends only on $d$ and $\Degf $), we see that 
the contribution from the second term in~\eqref{eq:bound k} (together with the implied 
constant) does not exceed $k^2/2$. Hence, for the choice of $R$ as in~\eqref{eq:R opt}  we have 
\begin{align*}
k^2 & \le k H^{d/2}  \varepsilon^{ -5d/2} R^{d+1/2+ o(1)} \\
& = k  H^{(3d+1)/2+o(1)}  \varepsilon^{-2d^2 \Degf  - d\Degf -7d/2 -1/2+o(1)}, 
\end{align*}
which concludes the proof. 

\section*{Acknowledgement} 

 The author is grateful to Kamil Bulinski and Alexander Fish for numerous conversations 
 and patient explanations, which have led to this work. Special thanks go to Alexander Fish for
 the very careful reading of the manuscript which helped to eliminate various imprecisions 
 in the initial version. 

This work was  supported   by ARC Grant~DP170100786.


\begin{thebibliography}{www}
 
\bibitem{BCS} R. Baker, C. Chen and I. E. Shparlinski,  `Bounds on the norms of maximal 
operators on Weyl sums', Preprint, 2021 (available at \url{https://arxiv.org/abs/2107.13674}).
 
\bibitem{BuFi} K.~Bulinski and A.~Fish,  `Glasner property for unipotently generated group actions on tori', 
 {\it Israel J. Math.} (to appear). 

\bibitem{CPR}  T. Cochrane, C. Pinner and J. Rosenhouse, 
`Sparse polynomial exponential sums',
{\it Acta Arith.\/}, {\bf 108} (2003),  37--52.

\bibitem{DiQi} P. Ding and M. G. Qi,  
`Further estimate of complete trigonometric sums', 
 {\it J. Tsinghua Univ.\/}, {\bf 29}(6) (1989),  74--85. 
 
 \bibitem{ErdSz}
 P. Erd{\H o}s and G. Szekeres, `{\"U}ber die Anzahl der Abelschen Gruppen gegebener Ordnung und {\"u}ber ein verwandtes zahlentheoretisches Problem',  
 {\it Acta sei. Math. Szeged VII\/}, {\bf 11} (1934), 95--102.
 
  \bibitem{Glas} 
S. Glasner, `Almost periodic sets and measures on the torus',  
 {\it Israel J. Math.\/}, {\bf 32} (1979),   161--172.
 
 \bibitem{KeLe} 
 M. Kelly and T. H. L\^e, `Uniform dilations in higher dimensions',
{\it J. Lond. Math. Soc.\/}, {\bf 88} (2013),   925--940.
 
\bibitem{LN} R. Lidl and H. Niederreiter,  {\it Finite Fields}, 
Cambridge Univ. Press, Cambridge, 1997.


\bibitem{Stech} S. B. Ste{\v{c}}kin, `An estimate of a complete
rational trigonometric sum', 
 {\it Trudy Mat. Inst. Steklov.\/}, {\bf 143}
(1977), 188--207 (in Russian).
 
\bibitem{Vau} R.C. Vaughan, {\it The Hardy-Littlewood method}, 
Cambridge University Press, 1997.

\end{thebibliography}
 \end{document}